\documentclass{jgcc} 

\pdfoutput=1

\usepackage{lastpage}

\jgccdoi{18}{1}{2}{17270}

\jgccheading{}{\pageref{LastPage}}{}{}{Jan.~12,~2026}{Apr.~1,~2026}{}   


\usepackage{hyperref}
\theoremstyle{plain} 

\newcommand{\floor}[1]{\lfloor #1 \rfloor}

\begin{document}

\title[On the Diophantine problem related to power circuits]{On the Diophantine problem related to power circuits}

\author[A.~Rybalov]{Alexander Rybalov}
\address{Sobolev Institute of Mathematics, Pevtsova 13, Omsk 644099, Russia.}	
\email{alexander.rybalov@gmail.com}  
\thanks{Supported by Russian Science Foundation, grant 25-11-20023}	





\begin{abstract}
  \noindent Myasnikov, Ushakov, and Won introduced power circuits in 2012 to 
construct a polynomial-time algorithm for the word
problem in the Baumslag group, which has a non-elementary Dehn function.
Power circuits are computational
structures that support addition and the operation
$(x,y) \mapsto x \cdot 2^y$ on integers.
They also posed the question of decidability of the
Diophantine problem over the structure
$\langle \mathbb{N}_{>0}; +, x \cdot 2^y, \leq, 1 \rangle$,
which is closely related to power circuits.
In this paper, we prove that the Diophantine problem over this structure is undecidable.
\end{abstract}

\maketitle

\begin{flushright}
{\it Dedicated to Alexei Miasnikov\\
on the occasion of his birthday.}
\end{flushright}

\section{Introduction}

Power circuits have been introduced by Myasnikov, Ushakov, and Won in \cite{MUW2} 
as computational structures supporting addition and the operation
$(x,y) \mapsto x \cdot 2^y$ on integers.
Using power circuits they constructed in \cite{MUW1} a polynomial-time algorithm for the word
problem in the Baumslag group
$$
G_{(1,2)} = \langle a,b\ |\ b^{-1} a^{-1} ba b^{-1} ab = a^2 \rangle,
$$
which has a non-elementary Dehn function.

Myasnikov, Ushakov, and Won posed in~\cite[Problem~10.3]{MUW2} the question
of decidability of the Diophantine problem over the structure
$$
\widetilde{N} = \langle \mathbb{N}_{>0}; +, x \cdot 2^y, \leq, 1 \rangle,
$$
which is closely related to power circuits. 
Here $\mathbb{N}$ denotes the set
of natural numbers including zero, and $\mathbb{N}_{>k}$ denotes the set of
natural numbers greater than $k$.

Recall that the Diophantine problem $\mathcal{DP}(\widetilde{N})$ over $\widetilde{N}$ is the algorithmic problem of deciding whether
a given equation or a system of equations over $\widetilde{N}$ has a solution.
The classical Diophantine problem $\mathcal{DP}(\mathbb{N})$ over the structure $\langle \mathbb{N}; +, \cdot, 1 \rangle$, known as 
Hilbert's tenth problem, is undecidable, as proved by Matiyasevich in \cite{Mat},
building on earlier work of Davis, Putnam, and Robinson in \cite{DPR}.
On the other hand, Semenov in \cite{Sem} proved decidability of the first-order
theory of natural numbers with addition and exponentiation $\langle \mathbb{N}; +, x \mapsto 2^x, 1 \rangle$. It follows that the Diophantine problem over this structure is decidable.

In this paper, we prove that the Diophantine problem over the structure
$\widetilde{N}$ is undecidable. As a consequence, this resolves
another question posed in~\cite[Problem~10.5]{MUW2}: ``Is $\widetilde{N}$ automatic?'' The answer is negative, since every
automatic structure has a decidable first-order theory~\cite{KN}, and hence a decidable Diophantine problem.

\section{Main result}

Consider the following variant of the Diophantine problem, denoted by
$\mathcal{DP}(\mathbb{N}_{>k})$, that  asks whether a given system of
equations over $\mathbb{N}$ admits a solution in $\mathbb{N}_{>k}$.

\begin{lem} \label{rest_dp}
For every natural number $k$ there is a Cook/Turing reduction from $\mathcal{DP}(\mathbb{N})$ to $\mathcal{DP}(\mathbb{N}_{>k})$.
\end{lem}
\begin{proof}
For a given system of Diophantine equations $S(x_1,\ldots,x_n)$ we 
enumerate all subsets $Y$ of $X = \{ x_1, \ldots, x_n \}$ and 
all possible assignments from $\{ 0, \ldots, k \}$ for the variables in $Y$, 
i.e., functions $f : Y \rightarrow \{ 0, \ldots, k \}$. 
Then for each $f$ we define the system $S_f$ by replacing each variable $y \in Y$ in $S$ with the value $f(y)$.
Denote the resulting finite set of systems by $A(S)$.
It is straightforward to verify that the system $S(x_1,\ldots,x_n)$ has a solution in $\mathbb{N}$ if and only if 
at least one system $S^{\prime} \in A(S)$ has a solution in $\mathbb{N}_{>k}$.
\end{proof}

\begin{cor} \label{dpk_undec}
For every natural number $k$ the problem $\mathcal{DP}(\mathbb{N}_{>k})$ is undecidable.
\end{cor}

Consider the structure $\widetilde{N} = \langle \mathbb{N}_{>0}; +, x \cdot 2^y, \leq, 1 \rangle$.
To prove undecidability of the Diophantine problem over $\widetilde{N}$ we reduce $\mathcal{DP}(\mathbb{N}_{>1})$ 
to $\mathcal{DP}(\widetilde{N})$.
For this purpose it is sufficient to define the multiplication of numbers greater than 1 in $\widetilde{N}$.

A relation $R \subseteq \mathbb{N}_{>1}^k$ is {\it Diophantine definable} in $\widetilde{N}$ if there exists a system of equations
(a conjunction of atomic formulas)
$S(y_1, \ldots, y_k, x_1, \ldots, x_n)$ over $\widetilde{N}$ such that
$$
\forall a_1 \ldots \forall a_k\ R(a_1, \ldots, a_k) \Leftrightarrow 
\exists x_1 \ldots \exists x_n\ S(a_1, \ldots, a_k, x_1, \ldots, x_n).
$$
Also, a function $f : \mathbb{N}_{>1}^k \rightarrow \mathbb{N}_{>1}$ is {\it Diophantine definable} in $\widetilde{N}$ if the graph of
a function $f$ is Diophantine definable in $\widetilde{N}$.

We use notation $a\ |\ b$ for $a,b \in \mathbb{N}$ if $a$ divides $b$.

\begin{lem} \label{power2}
For all $n,m \in \mathbb{N}$ the following holds:
$$
m\ |\ n \Leftrightarrow 2^m-1\ |\ 2^n - 1.
$$
\end{lem}
\begin{proof}
Suppose $m$ divides $n$ and $n = km$ for some $k \in \mathbb{N}$. Then
$$
2^n - 1 = 2^{km} - 1 = (2^m - 1)(2^{m (k-1)} + \ldots + 2^m + 1).
$$
Suppose $m$ does not divide $n$ and $n = km + r$ with natural $k$ and $0 < r < m$. Then
$$
2^n - 1 = 2^{km+r} - 2^r + 2^r - 1 = 2^r (2^{km} - 1) + 2^r - 1
$$
is not divisible by $2^m - 1$ since $2^m - 1$ divides $2^{km} - 1$ and $2^m - 1 > 2^r - 1 > 0$.
\end{proof}

\begin{lem} \label{div_def}
The divisibility relation $x\ |\ y$ is Diophantine definable in $\widetilde{N}$.
\end{lem}
\begin{proof}
By Lemma \ref{power2}
$$
x\ |\ y \Leftrightarrow \exists z\ 2^y-1 = z (2^x - 1) \Leftrightarrow \exists z\ 2^y + z = z \cdot 2^x + 1.
\eqno\qed
$$\let\qed\relax
\end{proof}

Robinson proved in \cite{Rob} that the first-order theory of natural numbers with addition and divisibility relation
is undecidable. But Beltjukov in \cite{Bel} and Lipshitz in \cite{Lip} proved that the Diophantine problem over this structure
is decidable. Therefore, the Diophantine definability of the divisibility relation is not sufficient to prove the undecidability of the Diophantine problem over $\widetilde{N}$ and we need further research.

\begin{lem} \label{str_def}
The strict order relation $x < y$ is Diophantine definable in $\widetilde{N}$.
\end{lem}
\begin{proof}
Note that
$x < y \Leftrightarrow \exists z\ x+z = y$.
\end{proof}

We use notation $\floor{a}$ to denote the integer part of a real number $a$.

\begin{lem} \label{log_def}
The integer binary logarithm $\floor{\log_2{x}}$ for $x>1$ is Diophantine definable in $\widetilde{N}$.
\end{lem}
\begin{proof}
Note that
$$
y = \floor{\log_2 {x}} \Leftrightarrow (2^y \leq x) \wedge (x < 2^{y+1}).
\eqno\qed
$$\let\qed\relax
\end{proof}

\begin{lem} \label{sqr_def}
The function $sq(x) = x^2$ for $x>1$ is Diophantine definable in $\widetilde{N}$.
\end{lem}
\begin{proof}
The set 
$$
S(x) = \{ k x(x+1)\ :\ k \in \mathbb{N} \}
$$
is Diophantine definable in $\widetilde{N}$, since
$$
y \in S(x) \Leftrightarrow (x\ |\ y) \wedge (x+1\ |\ y).
$$
Now consider the following Diophantine definable set in $\widetilde{N}$:
$$
S^{\prime}(x) = \{ y\ :\ y + x \in S(x),\ \floor{\log_2 {y}} \leq 2 \floor{\log_2 {x}} + 1 \}.
$$
If $k \geq 4$, then
$$
\floor{\log_2 ( k x(x+1) - x ) } = \floor{\log_2 (k x^2 + (k-1)x) } \geq \floor{\log_2 (k x^2) } = 
$$
$$
= \floor{ \log_2 {k} + 2 \log_2 x } \geq \floor{ 2 + 2 \log_2 {x} } \geq 2 + \floor{ 2 \log_2 x } \geq 2 + 2 \floor{ \log_2 x}.
$$
Hence, for every $x>1$ we have
$$
S^{\prime}(x) \subseteq \{ x^2, 2x^2+x, 3x^2+2x \}.
$$
Note that $x^2 \in S^{\prime}(x)$ since
$$
\floor{ \log_2 (x^2)} = \floor{2 \log_2 x} \leq 2 \floor{ \log_2 x} + 1.
$$
To ensure that $2x^2+x$ and $3x^2+2x$ do not belong to $S^{\prime}(x)$,
consider the following Diophantine over $\widetilde{N}$ condition:
$$
P(x, y) = (x+2\ |\ y + 2x) \wedge (x+3\ |\ y + 3x).
$$
Element $x^2$ satisfies this condition because $x+2$ divides $x^2+2x$ and $x+3$ divides $x^2+3x$.
On the other hand, $2x^2+x + 2x = (2x-1)(x+2) + 2$ does not satisfy this condition, because
it is not divisible by $x+2$ for all natural $x$.
Similarly, $3x^2+2x + 2x = (3x-2)(x+2) + 4$ does not belong to $P$, because
it is divisible by $x+2$ only for $x=2$, but for $x=2$ we see that $3x^2+2x + 3x = 22$
is not divisible by $x+3 = 5$.
\end{proof}

\begin{lem} \label{mult_def}
The multiplication operation $mul(x,y) = xy$ for $x,y>1$ is Diophantine definable in $\widetilde{N}$.
\end{lem}
\begin{proof}
Note that
$$
z = xy \Leftrightarrow 2z = (x+y)^2 - x^2 - y^2 \Leftrightarrow z + z + x^2 + y^2 = (x+y)^2.
\eqno\qed
$$\let\qed\relax
\end{proof}

\begin{thm} \label{main}
The Diophantine problem over $\widetilde{N} = \langle \mathbb{N}_{>0}; +, x \cdot 2^y, \leq, 1 \rangle$ is undecidable.
\end{thm}
\begin{proof}
We reduce $\mathcal{DP}(\mathbb{N}_{>1})$ to the Diophantine problem over $\widetilde{N}$ in the following way.
Without loss of generality any given nontrivial system $S$ of Diophantine equations over $\mathbb{N}_{>1}$ 
consists of equations of the form $P(x_1, \ldots, x_n) = Q(x_1, \ldots, x_n)$,
where $P$ and $Q$ are non-zero polynomials with positive integer coefficients.
Such system $S$ can be transformed to an equivalent system in the {\it Skolem form}, consisting of equations of
the following types:
\begin{enumerate}
\item $x_i = x_j x_k$,
\item $x_i = x_j + x_k$,
\item $x_i = x_j + 1$,
\item $x_i = x_j$.
\end{enumerate}
By Lemma \ref{mult_def}, we can replace every equation of type (1) by an equivalent system of equations over $\widetilde{N}$.
Also for every variable $x$, which is included in equations of types (2), (3) or (4), but not included in any equation of type (1),
we add the Diophantine condition $x>1$ using Lemma \ref{str_def}. 
By construction, the obtained system of equations over $\widetilde{N}$ is equivalent to the original system $S$ over $\mathbb{N}_{>1}$.

By Corollary \ref{dpk_undec}, the problem $\mathcal{DP}(\mathbb{N}_{>1})$ is undecidable;
therefore, $\mathcal{DP}(\widetilde{N})$ is also undecidable.
\end{proof}

Since every automatic structure has a decidable first-order theory~\cite{KN},
we obtain the following corollary of Theorem \ref{main}.

\begin{cor}
$\widetilde{N}$ is not automatic.
\end{cor}

\section*{Acknowledgment}
  \noindent The author thanks Alexander Ushakov for many useful suggestions and remarks.

\bibliographystyle{plain}
\bibliography{rybalov}

\end{document}